\documentclass[letterpaper,11pt]{article}
\usepackage[margin=1in]{geometry}
\usepackage{comment}
\usepackage{amsthm}
\usepackage{mathrsfs}
\usepackage{fullpage}
\usepackage{setspace}
\usepackage{url}
\usepackage{tikz}
\usetikzlibrary{fit, positioning}
\usepackage{mathtools}
\usetikzlibrary{er}
\usetikzlibrary{arrows.meta}
\usepackage{amsmath,amsfonts,amssymb}
\usepackage{verbatim,enumitem,graphics}
\usepackage{microtype}
\usepackage{xcolor}
\usepackage[backref=page]{hyperref}
\hypersetup{%
    colorlinks,
    linkcolor={red!50!black},
    citecolor={green!50!black},
    urlcolor={blue!50!black}
}
\usepackage{cleveref}
\usepackage{dsfont}
\usepackage[abbrev,msc-links,backrefs]{amsrefs} 
\usepackage{doi}
\usepackage{lineno}
\usepackage[normalem]{ulem}
\usepackage{mathdots}
\usepackage{algorithm,algpseudocode}

\newtheorem{theorem}{Theorem}[section]
\newtheorem{lemma}[theorem]{Lemma}
\newtheorem{proposition}[theorem]{Proposition}

\newtheorem{corollary}[theorem]{Corollary}

\newtheorem{claim}{Claim}
\theoremstyle{plain}
\theoremstyle{definition}
\newtheorem{definition}[theorem]{Definition}
\theoremstyle{plain}

\theoremstyle{remark}
\newtheorem{remark}[theorem]{Remark}

\let\epsilon=\varepsilon

\begin{document}

\setstretch{1.27}
%\linenumbers

\title{Regular bipartite decompositions of pseudorandom graphs}

\author{Asaf Ferber\thanks{Department of Mathematics, University of California, Irvine. 
Email:  {{\tt asaff@uci.edu}}. Research is supported by NSF grant DMS-1953799, NSF Career DMS-2146406, a Sloan's fellowship, and an Air force grant FA9550-23-1-0298.} \and Bryce~Frederickson\thanks{Department of Mathematics, Emory University. Email: {\tt bfrede4@emory.edu}. Research is supported by the NSF Graduate Research Fellowship Program under Grant No. 1937971.} \and Dingjia~Mao\thanks{Department of Mathematics, University of California, Irvine.
Email:  {{\tt dingjiam@uci.edu}}} \and Liana~Yepremyan\thanks{Department of Mathematics, Emory University. 
Email:  {{\tt liana.yepremyan@emory.edu}}. Research is supported by the National Science Foundation grant 2247013: Forbidden and Colored Subgraphs. } \and Yizhe~Zhu\thanks{Department of Mathematics, University of Southern California.
Email:  {{\tt yizhezhu@usc.edu}.} Research is supported by NSF-Simons Research Collaborations on the Mathematical and Scientific Foundations of Deep Learning and an AMS-Simons Travel Grant. 
}}

\date{}

\maketitle

\begin{abstract}
In 1972, Kotzig proved that for every even $n$, the complete graph $K_n$ can be decomposed into $\lceil\log_2n\rceil$ edge-disjoint regular bipartite spanning subgraphs, which is best possible. In this paper, we study regular bipartite decompositions of $(n,d,\lambda)$-graphs, where $n$ is an even integer and $d_0\leq d\leq n-1$ for some absolute constant $d_0$. With a randomized algorithm,
we prove that such an $(n,d,\lambda)$-graph with $\lambda\leq d/12$ can be decomposed into at most $\log_2 d + 36$ regular bipartite spanning subgraphs. This is best possible up to the additive constant term. As a consequence, we also improve the best known bounds on $\lambda = \lambda(d)$ by Ferber and Jain (2020) to guarantee that an $(n,d,\lambda)$-graph on an even number of vertices admits a $1$-factorization, showing that $\lambda \leq cd$ is sufficient for some absolute constant $c > 0$.

\end{abstract}
%\linenumbers

\section{Introduction}

\indent

Packing and covering problems have a long history in graph theory. A \emph{covering} of a graph $G$ is a collection $\mathcal C$ of subgraphs of $G$ such that $E(G) = \bigcup_{H \in \mathcal{C}} E(H)$. If the graphs in $H$ are edge-disjoint, then $\mathcal C$ is called a \emph{decomposition}, or \emph{packing}, of $G$. Typical problems in the area ask for the minimum size of a decomposition $\mathcal C$ (or covering) of a fixed graph $G$, subject to each subgraph in $\mathcal C$ having a certain form. For example, one might look for a decomposition (or covering) into matchings \cite{vizing1964estimate}, paths \cites{Lovasz1968, fan2002subgraph}, cycles \cites{Bucic2023, hilton1984hamiltonian, kuhn2014hamilton, fan2002subgraph}, triangles \cite{kirkman1847problem}, cliques \cite{erdos1966representation}, or bicliques \cite{graham1972embedding}.

In this paper, we consider decompositions of a graph $G$ into regular bipartite spanning subgraphs. 
This type of decomposition problem was first studied by Kotzig \cite{kotzig1972decompositions}, who proved the following exact result for the complete graph $K_n$ on an even number of vertices.

\begin{theorem}\label{thm:complete-graph}
Let $n \geq 2$ be an even integer. Then the edge set of the complete graph $K_n$ can be decomposed into $\lceil\log_2 n\rceil$ regular bipartite spanning subgraphs, which is best possible.
\end{theorem}

We aim to extend \Cref{thm:complete-graph} to a larger class of graphs. For a graph $G$ to admit a decomposition into regular bipartite spanning subgraphs, $G$ must be regular ($1$-factorizable, in fact) on an even number of vertices. Our main result concerns regular bipartite decompositions of regular \emph{pseudorandom} graphs. Specifically, we are interested in the class of \emph{$(n,d,\lambda)$-graphs}. These graphs play an important role in combinatorics and theoretical computer science, including sampling, complexity theory, and
the design of error-correcting codes \cites{alon1986eigenvalues,hoory2006expander,lubotzky2012expander}. These are defined in the following way. Given a graph $G$ on vertex set $V=\{v_1,\ldots,v_n\}$, its \emph{adjacency matrix} $A:=A(G)$ is an $n\times n$, $0/1$ matrix, defined by $A_{ij}=1$ if and only if $v_iv_j\in E(G)$. Let $\lambda_1\geq \lambda_2\geq\ldots \geq \lambda_n$ be the eigenvalues of $A$. It is a well-known fact that for a $d$-regular graph $G$, we always have $\lambda_1=d$ and $\lambda_n \geq -d$. We say that $G$ is an \emph{$(n,d,\lambda)$-graph} if it is a $d$-regular graph on $n$ vertices with $\lambda (G)\coloneqq \max\{|\lambda_2|,|\lambda_n|\}\leq \lambda$.

Decomposition problems in pseudorandom graphs have been intensively studied in recent years. For example, the \emph{perfect matching}, or \emph{$1$-factor}, which is a collection of vertex-disjoint edges covering all the vertices, is the most fundamental type of regular bipartite spanning subgraph. A decomposition of a graph into perfect matchings is called a \emph{$1$-factorization}. For finding a single perfect matching, Krivelevich and Sudakov \cite{krivelevich2006pseudo} showed that an $(n,d,\lambda)$-graph with an even number of vertices contains a perfect matching whenever $\lambda\leq d-2$. More precise spectral conditions to guarantee a perfect matching were later given by Cioab\v a, Gregory and Haemers \cite{cioabua2009matchings}. As for $1$-factorizations of $(n,d,\lambda)$-graphs, Ferber and Jain \cite{ferber20181} proved the following result.

\begin{theorem}[Theorem 1.1 in \cite{ferber20181}]\label{thm:1-factorization of pseudo}
For every $\varepsilon>0$, there exists $d_0,n_0\in\mathbb{N}$ such that for all even integers $n\geq n_0$ and for all $d\geq d_0$ the following holds. Suppose that $G$ is an $(n,d,\lambda)$-graph with $\lambda\leq d^{1-\varepsilon}$. Then the edge set of $G$ can be decomposed into $d$ edge-disjoint perfect matchings.
\end{theorem}

Another important decomposition problem which has received much attention is that of decomposing a graph into edge-disjoint \emph{Hamilton cycles}, which are  cycles passing through all the vertices. Hamilton cycles are one of the most studied objects in graph theory, and determining if a graph contains a single Hamilton cycle is an NP-complete problem. The existence of Hamilton cycles in $(n,d,\lambda)$-graphs was first studied by Krivelevich and Sudakov \cite{krivelevich2003sparse}, who proved that for sufficiently large $n$, any $(n,d,\lambda)$-graph with
\[
\lambda/d\leq \frac{(\log\log n)^2}{1000\log n(\log\log\log n)}
\]
has a Hamilton cycle. They also conjectured in the same paper that the condition can be reduced to $\lambda/d\leq c$ for some absolute constant $c>0$. After some intermediate results by Glock, Correia and Sudakov \cite{glock2023hamilton}, and Ferber, Han, Mao and Vershynin \cite{ferber2024hamiltonicity}, the conjecture was finally proved very recently by Dragani\'c, Montgomery,  Correia,  Pokrovskiy and Sudakov \cite{draganic2024hamiltonicity}. Regarding the corresponding decomposition problem, the best known result was given by K\"uhn and Osthus \cite{kuhn2013hamilton} (see also Theorem 1.11 in \cite{kuhn2014hamilton}), who proved that for any $\alpha>0$ and sufficiently large even integer $n$, any $(n,d,\lambda)$-graph with  $d\geq \alpha n$ and $\lambda\leq \varepsilon n$ for some constant $\varepsilon=\varepsilon(\alpha)$ can be decomposed into edge-disjoint Hamilton cycles.

We now state our main result, which says that for sufficiently large $d$, any $(n,d,\lambda)$-graph on an even number of vertices with $\lambda \leq \frac{d}{12}$ can be decomposed into $\log_2 d + O(1)$ regular bipartite spanning subgraphs. This assumes no relationship between $d$ and $n$ besides the trivial condition $d \leq n-1$, and the bound $\log_2 d + O(1)$ is best possible since any graph with chromatic number $\Omega(d)$, such as $K_{d+1}$, cannot be decomposed into fewer than $\log_2 d - O(1)$ bipartite subgraphs of any kind.

\begin{theorem}\label{main-thm}
    There exists a positive integer $d_0$ such that the following holds for all integers $d \geq d_0$. Let $n \geq d+1$ be an even integer, and let $G$ be an $(n,d,\lambda)$-graph on $[n]$ with $\lambda \leq \frac{d}{12}$. Then $G$ admits a decomposition into at most $\log_2 d + 36$ regular bipartite spanning subgraphs.
\end{theorem}

Our proof of Theorem~\ref{main-thm} can be turned into a randomized algorithm in expected polynomial time for decomposing pseudorandom regular graphs into bipartite spanning graphs, by using the algorithmic Lovász Local Lemma \cite{moser2010constructive} and the polynomial-time algorithm for finding $f$-factors introduced in \cite{anstee1985algorithmic}.

Since any regular bipartite graph can be decomposed into perfect matchings, \Cref{main-thm} also shows that the conclusion of Theorem \ref{thm:1-factorization of pseudo} holds for $\lambda/d \leq c$ for some absolute constant $c > 0$, answering a question in \cite{ferber20181}. 

\begin{corollary}\label{cor:1-factorization}
There exists an integer $d_0$ such that the following holds for every integer $d\geq d_0$ and every even integer $n \geq d+1$. Let $G$ be an $(n,d,\lambda)$-graph with $\lambda \leq \frac{d}{12}$. Then $G$ admits a $1$-factorization.
\end{corollary}

The paper is organized as follows. At the end of this section, we sketch the proof of our main result (\Cref{main-thm}). In \Cref{section:preliminary}, we collect some standard tools in both probability and graph theory. 
In \Cref{sec:random bisections}, we outline the probabilistic portion of our argument, which uses nothing more than Chernoff's bounds and the Lov\'asz Local Lemma. This argument can produce a decomposition of any almost $d$-regular graph into $\log_2 d + O(1)$ \textit{almost} regular bipartite spanning subgraphs.
Next, in \Cref{sec:regularization}, we show how to use the expansion properties of our $(n,d,\lambda)$-graph $G$ guaranteed by the condition $\lambda \leq \frac{d}{12}$ to ``regularize'' the decomposition obtained in the previous section to give a decomposition of $G$ into regular bipartite spanning subgraphs. Finally, we combine these results to finish the proof of our main theorem in \Cref{section:main proof}. 

\subsection{Notation}
\indent

Throughout the paper, all graphs are simple and finite with vertex set $[n]$ unless specified otherwise. For a graph $G=(V,E)$, we write $e(G):=|E(G)|$. For a vertex subset $U \subseteq V(G)$, we write $e_G(U) := e(G[U])$ to denote the number of edges in $G$ with both endpoints in $U$. For two (not necessarily disjoint) vertex sets $A, B \subseteq V (G)$, we define $E_G(A,B)$ to be the set of all ordered pairs~$(x,y) \in V(G) \times V(G)$ with $xy\in E(G)$, $x\in A$, and $y\in B$. 
We denote $e_G(A,B):=|E_G(A,B)|$. In particular, we have $e_G(U,U) = 2e_G(U)$ for any $U \subseteq V(G)$. For two disjoint subsets $X,Y\subseteq V(G)$, we write $G[X,Y]$ to denote the induced bipartite subgraph on $G$ with parts $X$ and $Y$. 
We write~$\delta(G)$ to denote the minimum degree of $G$, and we write $\Delta(G)$ to denote its maximum degree. The neighborhood of a vertex $v \in V(G)$, denoted $N_G(v)$, is the set of vertices adjacent to $v$. 
Similarly, for a digraph~$D = (V,E)$, we define $N_D^+(v)$ to be the out-neighborhood of a vertex $v$, consisting of those vertices $u$ for which~${(v,u) \in E(D)}$. 
We write $a=b\pm c$ as a shorthand for the double-sided inequality $b-c\leq a\leq b+c$. For the cardinality of a set $S$, we use~$\#S$ and~$|S|$ interchangeably. Lastly, we write $Z \sim \mathrm{Bin}(n,p)$ to denote that $Z$ is a binomially distributed random variable with $\mathbb P(Z = k) = \binom{n}{k}p^k(1-p)^{n-k}$ for every integer $k \geq 0$.

\subsection{Proof outline}\label{section:outline}
\indent

We use a randomized algorithm to decompose the $(n,d,\lambda)$-graph $G$ into $\log_2 d + O(1)$ regular bipartite spanning subgraphs. Roughly speaking, we first set aside a specially chosen ``absorber'', then we decompose the remaining graph into \textit{almost} regular bipartite spanning subgraphs. Our absorber will simply be an induced balanced bipartite subgraph $G[X,Y]$, where the bipartition~$\{X,Y\}$ is chosen randomly so that $G[X,Y]$ is almost $\frac{d}{2}$-regular. The success of our procedure relies on the following remarkable property of $G$: for this choice of~$\{X,Y\}$, the graph $G[X,Y]$ can ``regularize'' just about any pair $H_X$, $H_Y$ of almost regular bipartite spanning subgraphs of~$G[X]$ and~$G[Y]$, respectively (see Lemmas~\ref{lem:good induced bipartite subgraphs of expanders} and \ref{lem:regularization}), in the sense that we can robustly find a sparse subgraph~$R$ of $G[X,Y]$ such that $H_X \cup H_Y \cup R$ is regular and bipartite. Therefore, after obtaining suitable decompositions $\{H_{X,j}\}_j$ and $\{H_{Y,j}\}_j$ of $G[X]$ and $G[Y]$, respectively, we can find edge-disjoint subgraphs $\{R_j\}_j$ of $G[X,Y]$ so that each $H_{X,j} \cup H_{Y,j} \cup R_j$ is a regular bipartite spanning subgraph of $G$. Since $G$ was regular to begin with, the remaining subgraph $H \subseteq G[X,Y]$ will be regular as well, so $\{H_{X,j} \cup H_{Y,j} \cup R_j\}_j \cup \{H\}$ will be a decomposition of $G$ into regular bipartite spanning subgraphs.

To illustrate the regularization step in more detail, suppose that $H_X$ and $H_Y$ have respective bipartitions $\{X', X''\}$ and $\{Y', Y''\}$ such that $|X'| = |Y'|$, $|X''| = |Y''|$, and both~$G[X',Y']$ and~$G[X'',Y'']$ have minimum degree at least $\frac{d}{5}$. Further suppose that $H_X$ and $H_Y$ have the same number of edges. A theorem of Ore (\Cref{thm:Ore})  generalizing Hall's matching theorem gives a necessary and sufficient condition for the graphs~$G[X', Y']$ and $G[X'', Y'']$ to have respective subgraphs $R'$ and $R''$ with prescribed degrees so that $H_X \cup H_Y \cup R' \cup R''$ is regular. This condition is easily verified, with room to spare, by the expander mixing lemma (\Cref{lem:expander mixing lemma}), using the assumption that~$\lambda \leq \frac{d}{12}$. Note that $H_X \cup H_Y \cup R' \cup R''$ will also be bipartite, with bipartition~$\{X' \cup Y'', X'' \cup Y'\}$. See Figure~\ref{fig:regularization} for an illustration.
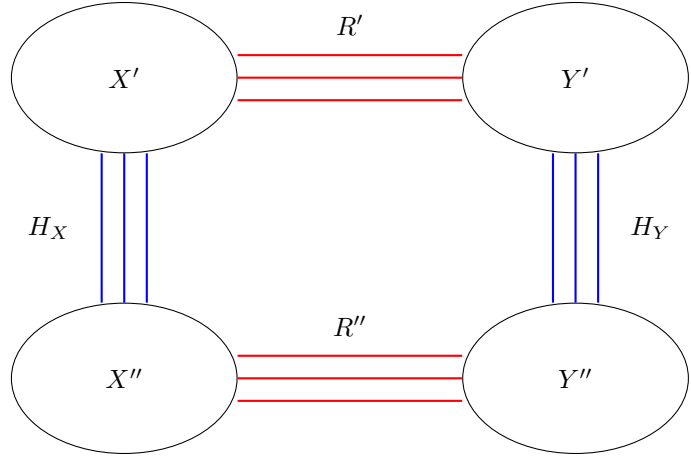
\begin{figure}
\centering
\begin{tikzpicture}[every node/.style={font=\small}]
    % Nodes for subgraphs with increased size
    \node[ellipse, draw, minimum width=3cm, minimum height=2cm] (X1) at (0, 4) {$X'$};
    \node[ellipse, draw, minimum width=3cm, minimum height=2cm] (Y1) at (6, 4) {$Y'$};
    \node[ellipse, draw, minimum width=3cm, minimum height=2cm] (X2) at (0, 0) {$X''$};
    \node[ellipse, draw, minimum width=3cm, minimum height=2cm] (Y2) at (6, 0) {$Y''$};

    % Labels for blue edge relations
    \node[left of=X1, yshift=-2cm] (Hx) {$H_{X}$};
    \node[right of=Y1, yshift=-2cm] (Hy) {$H_{Y}$};

    % Red Edges between X' and Y', and X'' and Y'' with labels
    \draw[red, thick] (X1.east) -- ++(0.6,0) -| (Y1.west);
    \draw[red, thick] ([yshift=0.3cm]X1.east) -- ++(0.6,0) -| ([yshift=0.3cm]Y1.west);
    \draw[red, thick] ([yshift=-0.3cm]X1.east) -- ++(0.6,0) -| ([yshift=-0.3cm]Y1.west);

  \node[right of=X1, xshift=2cm, yshift=0.7cm] {$R'$};
    
    \draw[red, thick] (X2.east) -- ++(0.6,0) -| (Y2.west);
    \draw[red, thick] ([yshift=0.3cm]X2.east) -- ++(0.6,0) -| ([yshift=0.3cm]Y2.west);
    \draw[red, thick] ([yshift=-0.3cm]X2.east) -- ++(0.6,0) -| ([yshift=-0.3cm]Y2.west);
\node[right of=X2, xshift=2cm, yshift=0.7cm] {$R''$};

    % Blue Edges between X' and X'', and Y' and Y'' with multiple lines
    \draw[blue, thick] (X1.south) -- ++(0, -0.6) -| (X2.north);
    \draw[blue, thick] ([xshift=-0.3cm]X1.south) -- ++(0, -0.6) -| ([xshift=-0.3cm]X2.north);
    \draw[blue, thick] ([xshift=0.3cm]X1.south) -- ++(0, -0.6) -| ([xshift=0.3cm]X2.north);
    
    \draw[blue, thick] (Y1.south) -- ++(0, -0.6) -| (Y2.north);
    \draw[blue, thick] ([xshift=-0.3cm]Y1.south) -- ++(0, -0.6) -| ([xshift=-0.3cm]Y2.north);
    \draw[blue, thick] ([xshift=0.3cm]Y1.south) -- ++(0, -0.6) -| ([xshift=0.3cm]Y2.north);
\end{tikzpicture}
\caption{Blue edges represent edges of $H_X$ and $H_Y$, and red edges represent edges in $R'$ and $R''$. $H_X\cup X_Y\cup R' \cup R''$ is a bipartite regular graph with bipartition $\{X'\cup Y'', X''\cup Y'\}$. }
\label{fig:regularization}
\end{figure}
Moreover, the graphs $R'$ and $R''$ can be found robustly, so that~$G[X,Y]$ can simultaneously regularize many pairs $H_X, H_Y$ in an edge-disjoint way. The fact that~$H_X$ and~$H_Y$ must have the same number of edges presents a minor technical issue, which we address in \Cref{sec:pairing}. Aside from this technicality, our problem is more or less reduced to finding randomized decompositions of $G[X]$ and $G[Y]$ into $\log_2 d + O(1)$ \textit{almost} regular bipartite spanning subgraphs each.

With the help of the Lov\'asz Local Lemma (Lemmas~\ref{lem:local-lemma} and \ref{lem:local-lemma-symmetric}), finding such decompositions is quite straightforward. To begin, our choice of $\{X,Y\}$ ensures that $G[X]$ and $G[Y]$ are each almost~$\frac{d}{2}$-regular. By iteratively taking cuts across random bisections of $X$ and $Y$, we can reduce the degree of each vertex by roughly a factor of $2$ with each step (see \Cref{lem:good bisection of one half}). After $\log_2 d - O(1)$ iterations, the remaining subgraphs $G_X \subseteq G[X]$ and $G_Y \subseteq G[Y]$ have constant maximum degree. With a constant number of additional random bisections, we can cover all of the edges in $G_X$ and~$G_Y$ by $O(1)$ bipartite subgraphs each (see \Cref{lem:cleanup}).

At last, we can use our absorber $G[X,Y]$ to simultaneously regularize all $\log_2 d + O(1)$ pairs of bipartite subgraphs covering $G[X] \cup G[Y]$ in an edge-disjoint way as described above, and thereby obtain our desired decomposition, completing the algorithm.

\section{Preliminaries}\label{section:preliminary}

\indent

In this section, we collect auxiliary results that will be used in the proof of our main results. 

\subsection{Probabilistic tools}
\indent

Throughout the paper, we will extensively use the following well-known Chernoff's bound (see, e.g., in \cite{alon2016probabilistic}).

\begin{theorem}[Chernoff's bounds]\label{chernoff}
	Let $X \sim \mathrm{Bin}\left(n,\frac{1}{2}\right)$. Then for any $t \geq 0$, we have
 \[\mathbb P\left(\left|X - \frac{n}{2}\right| \geq t\right) \leq 2\exp\left(-\frac{2t^2}{n}\right).\]
\end{theorem}

We will also make use of the following, known as the Lov\'asz Local Lemma (see, e.g., in \cite{alon2016probabilistic}). Before stating the lemma we need to introduce some notation: We say that a digraph $D$ (possibly with loops) is a \emph{dependency digraph} for a set $\mathcal A$ of events if $V(D) = \mathcal A$, and every event $A \in \mathcal A$ is mutually independent of the set $\mathcal A \setminus (\{A\} \cup N_D^+(A))$.

\begin{remark}\label{rem:dependencydigraph}
Let $\mathcal A$ be a set of events determined by a set $\mathcal Z = \{Z_1, \ldots, Z_k\}$ of independent random variables. Specifically, suppose that each event $A \in \mathcal A$ is determined by some subset~${\mathcal Z_A \subseteq \{Z_1, \ldots, Z_k\}}$. Then any digraph $D$ on $\mathcal A$ satisfying
\[\{A' : A' \neq A,\, \mathcal Z_{A'} \cap \mathcal Z_{A} \neq \emptyset\} \subseteq N_D^+(A)\]
for each event $A \in \mathcal A$ is a dependency digraph for $\mathcal A$.
\end{remark}

We are now ready to state the lemma. 

\begin{lemma}[Lov\'asz Local Lemma (asymmetric version)]\label{lem:local-lemma}
Let $(A_i)^n_{i=1}$ be a sequence of events in some probability space. Suppose that $D$ is a dependency graph for $(A_i)_{i=1}^n$, and suppose that there exist real numbers $(\alpha_i)^n_{i=1}$, such that $0 \leq \alpha_i < 1$, and
\[
\mathbb P\left(A_i\right)\leq \alpha_i\prod_{(A_i, A_j)\in E(D)}(1-\alpha_j).
\]
for all $1\leq i\leq n$. Then $\mathbb{P}\left(\bar{A}_i\right)\geq \prod_{i=1}^n(1-\alpha_i)$.
\end{lemma}

For certain applications, the following symmetric version of the Lov\'asz Local Lemma will be enough.

\begin{lemma}[Lov\'asz Local Lemma (symmetric version)]\label{lem:local-lemma-symmetric}
    Let $(A_i)^n_{i=1}$ be a sequence of events in some probability space. Suppose that $D$ is a dependency digraph for $(A_i)_{i=1}^n$ with maximum out-degree at most~$\Delta$. If~$
{\mathbb P\left(A_i\right)\leq \frac{1}{e(\Delta+1)}}$
for all $1\leq i\leq n$, then $\mathbb{P}\left(\bigcap_{i=1}^n\bar{A}_i\right)\geq \left(1-\frac{1}{\Delta+1}\right)^n$.
\end{lemma}

\subsection{Tools in graph theory}
\indent

One of the most useful tools in spectral graph theory is the expander mixing lemma, which asserts that an $(n,d,\lambda)$-graph has the following expansion property (see, e.g., \cite{hoory2006expander}).

\begin{lemma}[Expander mixing lemma]\label{lem:expander mixing lemma}
    Let $G$ be an $(n,d,\lambda)$-graph. Then for any two subsets~${S,T \subseteq V(G)}$, we have
    \[\left|e_G(S,T) -  \frac{d}{n}|S||T|\right| \leq \lambda \sqrt{|S||T|}.\]
\end{lemma}

Next, we collect some tools for factors in graphs. The next theorem, due to Vizing \cite{vizing1964estimate}, shows that every graph $G$ admits a proper edge coloring using at most $\Delta(G)+1$ colors.

\begin{theorem}[Vizing's theorem]\label{thm:Vizing}
Every graph with maximum degree $\Delta$ can be properly edge-colored with $k \in \{
\Delta,\Delta +1\}$ colors.
\end{theorem}

Just as a perfect matching (or $1$-factor) in a graph $G$ is a spanning subgraph in which every vertex has degree $1$, an \emph{$f$-factor} in $G$ is a spanning subgraph with degrees prescribed by some function $f$, as follows.

\begin{definition}[$f$-factor]\label{def: f-factor}
	Let $G$ be a graph, and let $f : V(G) \to \mathbb N$ be any function. An \emph{$f$-factor} in $G$ is a spanning subgraph $H$ of $G$ such that $\deg_H(v) = f(v)$ for every $v \in V(G)$.
\end{definition}

Generalizing Hall's theorem for perfect matchings (see, e.g., Theorem 3.1.11 in \cite{west2001introduction}), an equivalent condition for the existence of an $f$-factor in a bipartite graph was given by Ore \cite{ore1957graphs}.

\begin{theorem}[Ore's theorem]\label{thm:Ore}
	Let $H=(X\cup Y,E)$ be a bipartite graph, and let $f : V(H) \to \mathbb N$ be a function satisfying $f(X) = f(Y)$, where we use $f(S)$ to denote the sum $\sum_{v \in S} f(v)$ for $S \subseteq V(H)$. Then $H$ has an $f$-factor if and only if for every $S \subseteq X$ and $T \subseteq Y$, we have
	\begin{align}\label{eq:Ore}
 e_H(S,T) \geq f(S) + f(T) - f(X).
 \end{align}
\end{theorem}

\section{Almost regular decomposition via random bisections}\label{sec:random bisections}
\indent

The results of this section involve nothing more than routine applications of Chernoff's bounds (\Cref{chernoff}) and the Lov\'asz Local Lemma (Lemmas~\ref{lem:local-lemma} and \ref{lem:local-lemma-symmetric}). The basic underlying principle for each of these results is the following: in a random (nearly) balanced bipartition of the vertex set of an almost regular graph, one can ensure that the neighborhood of each vertex is nearly bisected as well. This principle is most simply illustrated by our first lemma.

\begin{lemma}\label{lem:almost regular bipartition}
	There exists a positive integer $d_0$ such that the following holds for all integers~$d \geq d_0$. Let $n \geq d+1$ be an even integer, and let $G$ be a $d$-regular graph on $[n]$. Then there exists a balanced bipartition $\{X,Y\}$ of $[n]$ such that for each graph $H \in \{G[X,Y], G[X], G[Y]\}$, we have~${\deg_H(v) = \frac{d}{2} \pm d^{2/3}}$ for every $v \in V(H)$. 
\end{lemma}
\begin{proof}
    We choose $\{X,Y\}$ randomly according to some distribution on the balanced bipartitions of~$[n]$. Rather than use a uniform distribution, we instead determine $\{X,Y\}$ by the outcome of $\frac n2$ independent random choices, which will give us enough independence to use the symmetric version of the Lov\'asz Local Lemma (\Cref{lem:local-lemma-symmetric}). Let $M = \{u_1v_1, \ldots, u_{n/2}v_{n/2}\}$ be any fixed perfect matching in the underlying complete graph $K_n$. For each $i = 1,\ldots, \frac{n}{2}$, independently select exactly one of $u_i$ or $v_i$ to be in $X$ and the other to be in $Y$, with equal probability $\frac{1}{2}$. We will refer to this random decision as $s_i$.
	
	For each vertex $v \in [n]$, let $A_v$ be the event that $|N_G(v) \cap X| \notin \frac{d}{2} \pm d^{2/3}$. It suffices to show that for sufficiently large $d$, we have with positive probability that no $A_v$ occurs.
	
	Fix a vertex $v$. Let $M_v \subseteq M$ be the set of pairs in $M$ with both endpoints in $N_G(v)$. Let~${M'_v \subseteq M}$ be the set of pairs in $M$ with exactly one endpoint in $N_G(v)$. Each pair from $M_v$ will contribute~$1$ to $|N_G(v) \cap X|$ with probability $1$, and each pair from $M'_v$ will contribute $1$ or $0$ independently with probability $\frac{1}{2}$. Thus $|N_G(v) \cap X| - |M_v| \sim \mathrm{Bin}\left(|M'_v|, \frac{1}{2}\right)$. Since $d = 2|M_v| + |M'_v|$, we have by Chernoff's bounds (\Cref{chernoff}) that
 \[\mathbb P(A_v) = \mathbb P\left(\left||N_G(v) \cap X| - |M_v| - \frac{|M'_v|}{2}\right| > d^{2/3}\right) \leq 2\exp\left(-\frac{2d^{4/3}}{|M'_v|}\right) \leq 2\exp\left(-2d^{1/3}\right).\]
	
	We now construct a dependency digraph $D$ for the events $\{A_v\}$. An event $A_v$ is determined by the those decisions $s_i$ for which $\{u_i,v_i\} \in M'_v$. By Remark~\ref{rem:dependencydigraph}, we can thus construct $D$ by including an edge from $A_v$ to $A_{v'}$ whenever there is a sequence $(v,u,u',v')$ of vertices in $[n]$ such that $vu, u'v' \in E(G)$, and either $u=u'$ or $\{u,u'\} \in M$. Note that $\Delta := \Delta^+(D) \leq 2d^2$, so for sufficiently large $d$, we have
 \[\mathbb P(A_v) \leq 2\exp\left(-2d^{1/3}\right) \leq \frac{1}{e(\Delta+1)}\]
 for every event $A_v$. By the symmetric version of the Lov\'asz Local Lemma (\Cref{lem:local-lemma-symmetric}), we have with positive probability that no event $A_v$ occurs. This completes the proof.
\end{proof}

Our next lemma follows the same principle, only now we have two graphs: one bipartite graph~$H$ with bipartition $\{X,Y\}$, and another graph $G_X$ contained entirely within the part $X$. We wish to bisect $X$ in such a way that every $G_X$-neighborhood is nearly bisected, and also so that the~$H$-neighborhood of every vertex in $Y$ is nearly bisected. The latter condition will come up frequently, so we make the following definition.

\begin{definition}\label{def:good bipartition}
    Let $H = (X \cup Y, E)$ be a balanced bipartite graph on $[n] = X \cup Y$, and let $d > 0$. We say that a bipartition $\{X', X''\}$ of $X$ is \emph{$d$-good} with respect to $H$ if the following two conditions hold:
    \begin{enumerate}[label=(G\arabic*)]
        \item\label{(G1)} $\left||X'| -|X''|\right| \leq 1$;
        \item\label{(G2)} $|N_H(y) \cap X'|, |N_H(y) \cap X''| \geq d$ for every $y \in Y$.
    \end{enumerate}
    We define a \emph{$d$-good} bipartition of $Y$ with respect to $H$ analogously.
\end{definition}

\begin{lemma}\label{lem:good bisection of one half}
    There exists a positive integer $d_0$ such that the following holds for all $d \geq d_0$. Let~$0 < \epsilon \leq \frac{5}{8}$, and let $2^{18} \leq d' \leq d$. Let $n$ be a positive even integer, and let $H = (X \cup Y, E)$ be a balanced bipartite graph on $[n]$ satisfying $\deg_H(v) = \frac{d}{2} \pm d^{2/3}$ for every $v \in [n]$. Let $G_X$ be a graph on $X$ satisfying $\deg_{G_X}(x) = (1 \pm \epsilon)d'$ for every $x \in X$. Then there exists a bipartition $\{X', X''\}$ of~$X$ such that the following two properties hold:
    \begin{enumerate}[label=(L\arabic*)]
        \item\label{(L1)} the graph $H_X := G_X[X', X'']$ satisfies 
        \[\deg_{H_X}(x) = \left(1 \pm \frac{2}{(d')^{1/3}}\right)\frac{\deg_{G_X}(x)}{2}\]
        for every $x \in X$.
        \item\label{(L2)} $\{X', X''\}$ is $\frac{d}{5}$-good with respect to $H$.
    \end{enumerate}
\end{lemma}
\begin{proof}
    We choose $\{X', X''\}$ randomly according to the following distribution, similar to the proof of \Cref{lem:almost regular bipartition}. Let 
    \[\tilde X := \left\{\begin{array}{l l}
    X &\text{if $|X|$ is even}, \\
    X \cup \{\tilde v\} & \text{if $|X|$ is odd},
    \end{array}\right.\]
    where $\tilde v \notin [n]$ is a dummy vertex. Let $M = \{u_1v_1, \ldots, u_{\lceil n/4 \rceil} v_{\lceil n/4 \rceil}\}$ be a matching in the underlying complete graph on $\tilde X$. For each $i = 1, \ldots, \left\lceil \frac{n}{4} \right\rceil$, we independently select exactly one of $u_i$ or $v_i$ to be in $X'$ and the other to be in $X''$, with equal probability $\frac{1}{2}$. As an exception to this rule, the dummy vertex $\tilde v$ is never added to $X'$ or $X''$, so we have that $\{X', X''\}$ is a bipartition of $X$.  
    We denote this random decision by $s_i$.
    
    For $x \in X$, let $A_x$ denote the event that $|N_{G_X}(x) \cap X'| \notin \left(1 \pm \frac{2}{(d')^{1/3}}\right)\frac{\deg_{G_X}(x)}{2}$, and for $y \in Y$, let~$B_y$ denote the event that $|N_H(y) \cap X'| \notin \frac{\deg_H(y)}{2} \pm \left(\frac{d}{20} - \frac{d^{2/3}}{2}\right)$. We will show using the asymmetric version of the Lov\'asz Local Lemma that with positive probability, none of these events occur. 

    For $x \in X$, let $M_x \subseteq M$ denote the set of pairs in $M$ with both endpoints in $N_{G_X}(x)$, and let~$M_x' \subseteq M$ denote the set of pairs in $M$ with exactly one endpoint in $N_{G_X}(x)$. Define $M_y$ and $M_y'$ analogously for $y \in Y$.

    Fix a vertex $x \in X$. Each pair from $M_x$ will contribute $1$ to $|N_{G_X}(x) \cap X'|$ with probability $1$, and each pair from $M'_x$ will contribute $1$ or $0$ independently with probability $\frac{1}{2}$. Thus 
    \[{|N_{G_X}(x) \cap X'| - |M_x| \sim \mathrm{Bin}\left(|M'_x|, \frac{1}{2}\right)}.\] Since $\deg_{G_X}(x) = 2|M_x| + |M'_x|$ and $\epsilon \leq \frac{5}{8}$, we have by Chernoff's bounds (\Cref{chernoff}) that
    \begin{align*}
        \mathbb P\left(A_x\right) &= \mathbb P\left(\left||N_{G_X}(x) \cap X'| - |M_x| - \frac{|M'_x|}{2}\right| > \frac{\deg_{G_X}(x)}{(d')^{1/3}}\right) \\
        &\leq 2\exp\left(-\frac{2(\deg_{G_X}(x))^2}{(d')^{2/3}|M'_x|}\right) \\
        &\leq 2\exp\left(-\frac{2\deg_{G_X}(x)}{(d')^{2/3}}\right) \\
        &\leq 2\exp\left(-\frac{3(d')^{1/3}}{4}\right).
    \end{align*}

    We similarly have for each $y \in Y$ that 
    \[|N_H(y) \cap X'| - |M_y| \sim \mathrm{Bin}\left(|M'_y|, \frac{1}{2}\right),\] and~${\deg_H(y) = 2|M_y| + |M'_y|}$, so
    \begin{align*}
        \mathbb P(B_y) &= \mathbb P\left(\left||N_H(y) \cap X'| - |M_y| - \frac{|M_y'|}{2}\right| > \frac{d}{20} - \frac{d^{2/3}}{2}\right)\\
        &\leq 2\exp\left(-\frac{2\left(\frac{d}{20} - \frac{d^{2/3}}{2}\right)^2}{\frac{d}{2} + d^{2/3}}\right) \\
        &\leq \exp\left(-\frac{d}{120}\right)
    \end{align*}
    for sufficiently large $d$.
    
    Note that each event $A_x$ is determined by those decisions $s_i$ for which $\{u_i, v_i\} \cap N_{G_X}(x) \neq \emptyset$, and each event $B_y$ is determined by those decisions $s_i$ for which $\{u_i, v_i\} \cap N_H(y) \neq \emptyset$. Therefore, since $G_X$ has maximum degree at most $(1+\epsilon)d'$, and $H$ has maximum degree at most $\frac{d}{2} + d^{2/3}$, it follows from \Cref{rem:dependencydigraph} that there exists a dependency digraph $D$ for $\{A_x : x \in X\} \cup \{B_y : y \in Y\}$ with
    \begin{align*}
        \#\{x' : (A_x, A_{x'}) \in E(D)\} &\leq 2(1+\epsilon)^2(d')^2 \leq 6(d')^2 & \text{for each $x \in X$}; \\
        \#\{y : (A_x, B_y) \in E(D)\} &\leq 2(1+\epsilon)d'\left(\frac{d}{2} + d^{2/3}\right) \leq 2d'd & \text{for each $x \in X$}; \\
        \#\{x : (B_y, A_x) \in E(D)\} &\leq 2(1+\epsilon)d'\left(\frac{d}{2} + d^{2/3}\right) \leq 2d'd & \text{for each $y \in Y$}; \\
        \#\{y' : (B_y, B_{y'}) \in E(D)\} &\leq 2\left(\frac{d}{2} + d^{2/3}\right)^2 \leq d^2 & \text{for each $y \in Y$}.
    \end{align*}
    Let $\alpha := \frac{1}{12(d')^2}$, and let $\beta := \frac{1}{d^3}$. For every $x \in X$, we have for sufficiently large $d$ that
    \begin{align*}
        \alpha(1-\alpha)^{6(d')^2}(1-\beta)^{2d'd} &\geq \alpha\left(1 - 6\alpha(d')^2\right)\left(1-2\beta d'd\right) \\
        &\geq \frac{1}{12(d')^2}\left(1 - \frac{1}{2}\right)\left(1 - \frac{2}{d}\right) \\
        &\geq 2\exp\left(-\frac{3(d')^{1/3}}{4}\right) \\
        &\geq \mathbb P(A_x),
    \end{align*}
    since $2^{18} \leq d' \leq d$. We also have for every $y \in Y$ and sufficiently large $d$ that
    \begin{align*}
        \beta(1-\alpha)^{2d'd}(1-\beta)^{d^2} &\geq \beta\left(1-2\alpha d' \right)^d\left(1-\beta d^2\right) \\
        &\geq \frac{1}{d^3}\left(1-\frac{1}{6 \cdot 2^{18}}\right)^d\left(1 - \frac{1}{d}\right) \\
        &\geq \exp\left(-\frac{d}{120}\right) \\
        &\geq \mathbb P(B_y).
    \end{align*}
    Therefore, by the asymmetric version of the Lov\'asz Local Lemma (\Cref{lem:local-lemma}), we have with positive probability that no event $A_x$ nor $B_y$ occurs. The fact that no $A_x$ occurs ensures \ref{(L1)}, and the fact that no $B_y$ occurs ensures \ref{(G2)} with $\frac{d}{5}$ in place of $d$. Note that \ref{(G1)} is satisfied automatically by design, so \ref{(L2)} is satisfied. This completes the proof.
\end{proof}

After applying \Cref{lem:almost regular bipartition} to bisect the vertex set of our $(n,d,\lambda)$-graph $G$ into $\{X,Y\}$, \Cref{lem:good bisection of one half} can be iterated to efficiently pull out edge-disjoint almost regular bipartite subgraphs from $G[X]$ until the remaining subgraph $G_X$ of $G[X]$ has constant maximum degree. The following lemma allows us to decompose this remainder $G_X$ into a constant number of bipartite subgraphs, each with a bipartition which is $\frac{d}{5}$-good with respect to $G[X,Y]$. To accomplish this, we take constantly many independent random bisections~$\left\{\{X'_j, X''_j\}\right\}_j$ of $X$ so that each vertex in $Y$ has large $G$-degree in each~$X'_j$ and~$X''_j$, and also so that each $G_X$-edge crosses at least one of the bisections $\{X'_j, X''_j\}$. Thus there exists a decomposition of $G_X$ into bipartite subgraphs with these bipartitions. Since each bipartite subgraph obtained in this way will have bounded maximum degree, we do not need to concentrate the $G_X$-degrees in these bisections in order to be able to apply our regularization step in \Cref{sec:regularization} (specifically, \ref{(A1)} in \Cref{lem:regularization} will be satisfied automatically).

\begin{lemma}\label{lem:cleanup}
    There exists a positive integer $d_0$ such that the following holds for all $d \geq d_0$. Let $n$ be a positive even integer, and let $H = (X \cup Y, E)$ be a balanced bipartite graph on $[n]$ satisfying $\deg_H(v) = \frac{d}{2} \pm d^{2/3}$ for every $v \in [n]$. Let $G_X$ be a graph on $X$ with maximum degree at most $\Delta$, where $1 \leq \Delta \leq d$. Then there exist $k := \lceil \log_2 \Delta \rceil + 8$ bipartitions $\{X'_1, X''_1\}, \ldots, \{X'_k, X''_k\}$ of $X$ such that the following two properties hold:
    \begin{enumerate}[label=(M\arabic*)]
        \item\label{(M1)} for every edge $e \in E(G_X)$, there is at least one index $j \in [k]$ such that $e$ crosses $\{X'_j, X''_j\}$;
        \item\label{(M2)} for each $1 \leq j \leq k$, $\{X'_j, X''_j\}$ is $\frac{d}{5}$-good with respect to $H$.
    \end{enumerate}
\end{lemma}
\begin{proof}
    We choose the bipartitions $\{X'_1, X''_1\}, \ldots, \{X'_k, X''_k\}$ independently at random according to the distribution described in the proof of \Cref{lem:good bisection of one half}. Specifically, we introduce a dummy vertex~${\tilde v \notin [n]}$, and we define
    \[\tilde X := \left\{\begin{array}{l l}
    X &\text{if $|X|$ is even}, \\
    X \cup \{\tilde v\} & \text{if $|X|$ is odd}.
    \end{array}\right.\]
    We fix a perfect matching $M = \{u_1v_1, \ldots, u_{\lceil n/4 \rceil} v_{\lceil n/4 \rceil}\}$ in the underlying complete graph on $\tilde X$, and for each $i = 1, \ldots, \left\lceil \frac{n}{4} \right\rceil$ and each $j = 1, \ldots, k$, we independently select exactly one of $u_i$ or $v_i$ to be in $X'_j$ and the other to be in $X''_j$ (excepting $\tilde v$), with equal probability $\frac{1}{2}$.  
    We denote this random decision by $s_{i,j}$.
    
    For $e \in E(G_X)$, let $A_e$ be the event that \ref{(M1)} fails for $e$. For $y \in Y$ and $j \in [k]$, let $B_{y,j}$ be the event that $|N_H(y) \cap X_j'| \notin \frac{\deg_H(y)}{2} \pm \left(\frac{d}{20} - \frac{d^{2/3}}{2}\right)$. For every $j \in [k]$ and $e \in E(G_X)$, the probability that $e$ crosses $\{X'_j, X''_j\}$ is either $1$ or $\frac{1}{2}$, depending on whether $e \in M$ or not. Thus we have
    \[\mathbb P(A_e) \leq 2^{-k} \leq \frac{1}{256\Delta}\]
    for every $e \in E(G_X)$, and
    \[\mathbb P(B_{y,j}) \leq \exp\left(-\frac{d}{120}\right)\]
    for every $y \in Y$ and $j \in [k]$ (see the proof of \Cref{lem:good bisection of one half}).

    Note that each event $A_e$ is determined by those decisions $s_{i,j}$ for which $\{u_i, v_i\} \cap e \neq \emptyset$, and each $B_{y,j}$ is determined by those decisions $s_{i,j}$ for which $\{u_i, v_i\} \cap N_H(y) \neq \emptyset$. Therefore, since $G_X$ has maximum degree at most $\Delta$, and $H$ has maximum degree at most $\frac{d}{2} + d^{2/3}$, it follows from \Cref{rem:dependencydigraph} that there exists a dependency digraph $D$ for $\{A_e : e \in E(G_X)\} \cup \{B_{y,j} : y \in Y,\, j \in [k]\}$ with
    \begin{align*}
        \#\{e' : (A_e, A_{e'}) \in E(D)\} &\leq 4\Delta & \text{for each $e \in E(G_X)$}; \\
        \#\{(y,j) : (A_e, B_{y,j}) \in E(D)\} &\leq 4\left(\frac{d}{2} + d^{2/3}\right)k \leq 4dk & \text{for each $e \in E(G_X)$}; \\
        \#\{e : (B_{y,j}, A_e) \in E(D)\} &\leq 2\left(\frac{d}{2} + d^{2/3}\right)\Delta \leq 2\Delta d & \text{for each $y \in Y$, $j \in [k]$}; \\
        \#\{(y',j') : (B_{y,j}, B_{y',j'}) \in E(D)\} &\leq 2\left(\frac{d}{2} + d^{2/3}\right)^2 \leq d^2 & \text{for each $y \in Y$, $j \in [k]$}.
    \end{align*}
    
    Set $\alpha := \frac{1}{250\Delta}$ and $\beta := \frac{1}{d^3}$. For this choice of parameters, we have for sufficiently large $d$ that
    \begin{align*}
        \alpha(1-\alpha)^{4\Delta}(1-\beta)^{4dk} &\geq \alpha(1-4\alpha\Delta)(1-4\beta dk) \\
        &\geq \frac{1}{250\Delta} \cdot \frac{123}{125}\left(1 - \frac{4k}{d^2}\right) \\
        &\geq \mathbb P(A_e)
    \end{align*}
    for each $e \in E(G_X)$, and
    \begin{align*}
        \beta(1-\alpha)^{2\Delta d}(1-\beta)^{d^2} &\geq \beta(1-2\alpha \Delta)^d(1-\beta d^2) \\
        &= \frac{1}{d^3}\left(\frac{124}{125}\right)^d\left(1 - \frac{1}{d}\right) \\
        &\geq \frac{1}{2d^3}\exp\left(-\frac{d}{124}\right) \\
        &\geq \mathbb P(B_{y,j})
    \end{align*}
    for each $y \in Y$, $j \in [k]$. Here we used that $1-x = \left(1 + \frac{x}{1-x}\right)^{-1} \geq \exp\left(-\frac{x}{1-x}\right)$ for every $x < 1$. Therefore, by the asymmetric version of the Lov\'asz Local Lemma (\Cref{lem:local-lemma}), we have with positive probability that none of the events $A_e$ nor $B_{y,j}$ occur, so both \ref{(M1)} and \ref{(M2)} hold.
\end{proof}

\section{Regularization}\label{sec:regularization}

\subsection{Pairing subgraphs with the same number of edges}\label{sec:pairing}
\indent

In this subsection, we prove two simple lemmas which allow us to fine-tune the number of edges each bipartite subgraph has in our decomposition at different stages. Recall that we first take a bisection $\{X,Y\}$ of the vertex set of our $(n,d,\lambda)$-graph $G$. By iteratively applying \Cref{lem:good bisection of one half}, then finally applying \Cref{lem:cleanup} to leftover graphs $G_X$ and $G_Y$ of constant maximum degree, we obtain decompositions $\{H_{X,j}\}_j$ and $\{H_{Y,j}\}_j$ of $G[X]$ and $G[Y]$ into almost regular bipartite spanning subgraphs of $G$. Our next two lemmas will be used to ensure that $e(H_{X,j}) = e(H_{Y,j})$ for each $j$. The first one will be applied repeatedly, after each iterative application of \Cref{lem:good bisection of one half}, while the second one will be applied once at the very end to the decompositions of $G_X$ and $G_Y$ obtained via \Cref{lem:cleanup}. These applications can be viewed as ``correcting'' steps that will make sure that the decompositions we obtain satisfy the necessary properties (specifically \ref{(A2)}) to apply our main regularization lemma (\Cref{lem:regularization}) and obtain the final desired decomposition.

\begin{lemma}\label{lem:subgraph with m edges}
    Let $H$ be a nonempty graph, and let $m \leq e(H)$ be a positive integer. Then there exists an $m$-edge spanning subgraph $H' \subseteq H$ satisfying the following properties:
    \begin{enumerate}[label=(S\arabic*)]
        \item\label{(S1)} $\Delta(H') - \delta(H') \leq \Delta(H) - \delta(H) + 2$;
        \item\label{(S2)} $\Delta(H \setminus H') \leq (\Delta(H) + 1)\frac{e(H) - m}{e(H)} + 1.$
    \end{enumerate}
\end{lemma}
\begin{proof}
    If $e(H) = m$ already, then we simply take $H' := H$. Otherwise, by Vizing's theorem (\Cref{thm:Vizing}), there exists a partition $E(H) = \bigcup_{i=1}^{\Delta+1} M_i$ of $E(H)$ into $\Delta + 1$ matchings~$M_1, \ldots, M_{\Delta+1}$, where $\Delta := \Delta(H)$. We order the matchings so that $|M_1| \geq \cdots \geq |M_{\Delta+1}|$. Let~$t$ be the minimum positive integer such that~${\sum_{i=1}^t |M_i| \geq e(H) - m}$. Let $M_t'$ be a submatching of~$M_t$ such that~${\sum_{i=1}^{t-1}|M_i| + |M'_t| = e(H) - m}$. Let $H'$ be the spanning subgraph of $H$ with edge set~${E(H') := (M_t \setminus M_t') \cup \bigcup_{i=t+1}^{\Delta+1}M_i}$. Then $e(H') = m$, as desired.

    Note that for every $v \in V(H)$, we have
    \[\deg_{H'}(v) \geq \deg_H(v) - t \geq \delta(H) - t\]
    since $E(H \setminus H')$ is a union of $t$ matchings. On the other hand, we have
    \[\deg_{H'}(v) \leq \Delta(H)+2-t\]
    since $E(H')$ is itself a union of $\Delta+2-t$ matchings. This verifies \ref{(S1)}.

    To prove \ref{(S2)}, note that if $t > 1$, then $\sum_{i=1}^{t-1} |M_i| < e(H) - m$ implies that $|M_i| \leq \frac{e(H)-m}{t-1}$ for every $i \geq t-1$. Thus
    \[e(H) = \sum_{i=1}^{t-1}|M_i| + \sum_{i=t}^{\Delta+1}|M_i| \leq (\Delta+1)\frac{e(H)-m}{t-1}.\]
    Rearranging gives
    \begin{equation}\label{eq:number of matchings}
        \Delta(H \setminus H') \leq t \leq (\Delta + 1)\frac{e(H) - m}{e(H)}+1,
    \end{equation}
    as required. Note that \eqref{eq:number of matchings} holds if $t=1$ as well. This completes the proof.
\end{proof}

\begin{proposition}\label{lem:same size refinements}
    Let $A$ and $B$ be nonempty finite sets of the same cardinality, and let~${\ell \geq 2}$. Let~$\mathcal P$ and $\mathcal Q$ be partitions of $A$ and $B$, respectively, with $|\mathcal P| + |\mathcal Q| = \ell$. Then there exist refinements~${\mathcal P' = \{S_1, \ldots, S_t\}}$ of~$\mathcal P$ and~${\mathcal Q' = \{T_1, \ldots, T_t\}}$ of~$\mathcal Q$ such that $t \leq \ell - 1$, and  $|S_i| = |T_i|$ for every $i = 1, \ldots, t$.
\end{proposition}
\begin{proof}
    We prove by induction on $\ell$. If $\ell = 2$, then we can take $\mathcal P' = \mathcal P = \{A\}$ and $\mathcal Q' = \mathcal Q = \{B\}$.
    
    Now assume that $\ell > 2$. Let $S$ be a set in $\mathcal P \cup \mathcal Q$ with minimum cardinality, say $S \in \mathcal P$ without loss of generality. Choose any set $T \in \mathcal Q$, and find $T' \subseteq T$ with $|T'| = |S|$. Apply the inductive hypothesis to
    \[\mathcal P_1 := \mathcal P \setminus \{S\} \qquad \text{and} \qquad \mathcal Q_1 := \left\{\begin{array}{l l}
    (\mathcal Q \setminus \{T\}) \cup \{T \setminus T'\} & \text{if } T' \neq T \\
    \mathcal Q \setminus \{T\} & \text{if } T' = T
    \end{array}\right.\]
    to obtain refinements $\mathcal P_1'$ of $\mathcal P_1$ and $\mathcal Q_1'$ of $\mathcal Q_1$ with parts of equal size and $|\mathcal P_1'| = |\mathcal Q_1'| \leq \ell-2$. Now $\mathcal P_1' \cup \{S\}$ and $\mathcal Q_1' \cup \{T'\}$ are the desired refinements of $\mathcal P$ and $\mathcal Q$. The proof is complete. 
\end{proof}

\subsection{Obtaining a regular decomposition from almost regular decompositions}

\indent

After taking a typical balanced bipartition $\{X,Y\}$ of a $d$-regular graph $G$, the tools developed up to this point allow us to obtain decompositions $\{H_{X,j}\}_{j=1}^K$ of $G[X]$ and $\{H_{Y,j}\}_{j=1}^K$ of $G[Y]$ into~${K = \log_2 d + O(1)}$ almost regular bipartite spanning subgraphs, with $e(H_{X,j}) = e(H_{Y,j})$ for each~$j \in [K]$. We can also ensure that each of these subgraphs admits a bipartition which is $\frac{d}{5}$-good with respect to $G[X,Y]$.

In the remainder of this section, we develop some additional tools to use edges from $G[X,Y]$ to regularize these decompositions of $G[X]$ and $G[Y]$ and produce our desired decomposition of $G$. The main two lemmas of this section (\Cref{lem:good induced bipartite subgraphs of expanders} and \Cref{lem:regularization}) will be used to show that if $G$ is an $(n,d,\lambda)$-graph with $\lambda \leq \frac{d}{12}$, then the induced bipartite subgraph $G[X,Y]$ contains edge-disjoint subgraphs~$R_1, \ldots, R_K$ such that each~$H_{X,j} \cup H_{Y,j} \cup R_j$ is a regular bipartite spanning subgraph of $G$. Thus~$\left\{H_{X,j} \cup H_{Y,j} \cup R_j\right\}_{j=1}^K \cup \left\{G[X,Y] \setminus \bigcup_{j=1}^K R_j\right\}$ will be our desired decomposition, completing the proof of \Cref{main-thm}.

For this, we make use of the following definitions.

\begin{definition}
    Let $H = (X \cup Y, E)$ be a balanced bipartite graph, and let $d,\rho, \alpha, \gamma \geq 0$. We say that $H$ is \emph{$(d, \rho, \alpha, \gamma)$-robustly matchable} if for every subgraph $F \subseteq H$ of maximum degree at most~$\rho d$, and for every positive $\alpha' \leq \alpha$, the bipartite graph $H \setminus F$ has an $f$-factor for every integer-valued function $f : X \cup Y \to [(1-\gamma)\alpha' d, \alpha' d]$ satisfying $f(X) = f(Y)$.
\end{definition}
Note that in particular, any balanced bipartite graph $H$ which is $(d,\rho,\alpha,\gamma)$-robustly matchable with $\alpha d \geq 1$ contains a perfect matching, but clearly this property is much stronger than that. We now add yet another layer of flexibility.

\begin{definition}
    Let $H = (X \cup Y, E)$ be a balanced bipartite graph, and let $d,\rho, \alpha, \gamma \geq 0$. We say that $H$ is \emph{$(d, \rho, \alpha,\gamma)$-good} if the following property holds: for any $X' \subseteq X$ and $Y' \subseteq Y$ satisfying~$|X'| = |Y'|$ and $\delta(H[X',Y']) \geq \frac{d}{5}$, we have that $H[X',Y']$ is $(d, \rho, \alpha, \gamma)$-robustly matchable.
\end{definition}

\begin{lemma}\label{lem:good induced bipartite subgraphs of expanders}
	Let $\rho, \alpha,\gamma > 0$ with $\rho + \alpha \leq \frac{1}{100}$ and $\gamma \leq \frac{1}{30}$.
	Let $G$ be an $(n,d,\lambda)$-graph with $n$ even and $\lambda \leq \frac{d}{12}$. Then for any balanced bipartition $\{X,Y\}$ of $[n]$, the induced bipartite subgraph $G[X,Y]$ is $(d,\rho,\alpha,\gamma)$-good.
\end{lemma}
\begin{proof}
	Let $F \subseteq G[X,Y]$ be a subgraph with maximum degree at most $\rho d$. Let $X' \subseteq X$ and $Y' \subseteq Y$ satisfying $|X'| = |Y'|$ and $\delta(G[X',Y']) \geq \frac{d}{5}$. Let $\alpha' \leq \alpha$, and let $f : X' \cup Y' \to [(1-\gamma)\alpha' d, \alpha' d]$ be an integer-valued function satisfying $f(X') = f(Y')$.
	
	Suppose that $G[X',Y'] \setminus F$ does not have an $f$-factor. Then by Ore's theorem (\Cref{thm:Ore}), there exist subsets $S \subseteq X'$ and $T \subseteq Y'$ such that
	\begin{equation}\label{eq:Ore contrapositive}
		e_{G \setminus F}(S,T) < f(S)+f(T)-f(X') = f(T) - f(X' \setminus S) \leq \alpha' d(|T| - (1-\gamma)|X' \setminus S|) \leq \alpha' d|T|.
	\end{equation}
	We assume without loss of generality that $|S| \geq |T|$.
	
	Suppose for a moment that $|T| \geq \frac{n}{10}$. Then $|S| \geq \frac{n}{10}$ as well, so by the expander mixing lemma (\Cref{lem:expander mixing lemma}), we have
	\[e_G(S,T) \geq \frac{d}{n}|S||T| - \lambda \sqrt{|S||T|} \geq \left(\frac{1}{n} - \frac{1}{12\sqrt{|S||T|}}\right)d|S||T| \geq \left(\frac{1}{n} - \frac{5}{6n}\right)\frac{dn}{10}|T| = \frac{d}{60}|T|.\]
	Therefore, we have
	\[e_{G \setminus F}(S,T) = e_G(S,T) - e_F(S,T) \geq \frac{d}{60}|T| - \rho d|T| \geq \alpha d |T|\]
	since $\Delta(F) \leq \rho d$ and $\rho + \alpha \leq \frac{1}{100}$, and this contradicts \eqref{eq:Ore contrapositive}. We conclude that $|T| \leq \frac{n}{10}$.
	
	Recall that $(1-\gamma)|X' \setminus S| < |T|$ by \eqref{eq:Ore contrapositive}. By the expander mixing lemma (\Cref{lem:expander mixing lemma}), we have
	\begin{align*}
	    e_{G \setminus F}(X' \setminus S, T) &\leq e_G(X' \setminus S, T) \\
     &\leq \frac{d}{n}|X' \setminus S||T| + \lambda \sqrt{|X' \setminus S||T|} \\
     &\leq \frac{1}{1-\gamma}\left(\frac{|T|}{n} + \frac{1}{12}\right)d|T| \\
     &< \frac{19d}{100}|T|
	\end{align*}
    since $\gamma \leq \frac{1}{30}$.
	On the other hand, we have
	\begin{align*}
		e_{G\setminus F}(X', T) &\geq \sum_{v \in T} \left(\deg_{G[X',Y']}(v) - \deg_F(v)\right) \\
		&\geq \frac{d}{5}|T| - \rho d|T| \\
		&\geq \frac{19d}{100}|T| + \alpha d |T| \\
		&> e_{G \setminus F}(X' \setminus S, T) + e_{G \setminus F}(S, T)\\
        &=e_{G\setminus F}(X',T)
	\end{align*}
	since $\rho + \alpha \leq \frac{1}{100}$. This is clearly impossible, and the proof is complete.
\end{proof}

\begin{lemma}\label{lem:regularization}
    Let $n$ be an even integer, and let $d, \rho, \alpha, \gamma > 0$. Let $G$ be a regular graph on $[n]$. Let $\{X,Y\}$ be a balanced bipartition of $[n]$, and suppose that $G[X,Y]$ is $(d,\rho,\alpha,\gamma)$-good. Further suppose that for some integer $K$ with $\frac{\rho}{\alpha}\leq K \leq \frac{\gamma \rho d}{2}$, there exist bipartite decompositions $\{H_{X,j}\}_{j \in [K]}$ and $\{H_{Y,j}\}_{j \in [K]}$ of~$G[X]$ and~$G[Y]$, respectively, such that the following conditions hold for each $j \in [K]$:
	\begin{enumerate}[label=(A\arabic*)]
		\item\label{(A1)} for each $U \in \{X,Y\}$, we have $\Delta(H_{U,j}) - \delta(H_{U,j}) \leq \frac{\gamma \rho d}{4K};$
        \item\label{(A2)} $e(H_{X,j}) = e(H_{Y,j})$;
        \item\label{(A3)} $H_{X,j}$ and $H_{Y,j}$ have respective bipartitions $\{X_j', X_j''\}$ and $\{Y_j', Y_j''\}$ which are $\frac{d}{5}$-good with respect to $G[X,Y]$.
	\end{enumerate}
    Then there exists a decomposition of $G$ into $K+1$ regular bipartite spanning subgraphs.
\end{lemma}
\begin{proof}
    For each $j = 1, \ldots, K$, in turn, we obtain a pair of spanning subgraphs~${R'_j \subseteq G[X'_j,Y'_j]}$ and~${R''_j \subseteq G[X''_j, Y''_j]}$ such that $H_{X,j} \cup H_{Y,j} \cup R'_j \cup R''_j$ is a regular bipartite spanning subgraph of $G$, as follows. All the while, we maintain that the graphs $R'_j \cup R''_j$ are edge-disjoint, and that~$\Delta(R'_j \cup R''_j) \leq \frac{\rho d}{K}$.

    By relabeling if necessary, we may assume by \ref{(G1)} that $|X'_j| = |Y'_j| = \left \lceil \frac{n}{4} \right \rceil$ and $|X''_j| = |Y''_j| = \left \lfloor \frac{n}{4} \right \rfloor$. Set $C := \left \lceil (1-\gamma)\frac{\rho d}{K} \right \rceil$. Let $f_j' : X'_j \cup Y'_j \to \mathbb N$ be given by
    \[f'_j(v) = C + \Delta(H_{X,j} \cup H_{Y,j}) - \deg_{H_{X,j} \cup H_{Y,j}}(v)\]
    for every $v \in X'_j \cup Y'_j$, and let $f_j'' : X''_j \cup Y''_j \to \mathbb N$ be given by
    \[f''_j(v) = C + \Delta(H_{X,j} \cup H_{Y,j}) - \deg_{H_{X,j} \cup H_{Y,j}}(v)\]
    for every $v \in X''_j \cup Y''_j$. We wish to find an $f'_j$-factor $R'_j$ of~$(G \setminus F_j)[X'_j, Y'_j]$ and an $f''_j$-factor $R''_j$ of~$(G \setminus F_j)[X''_j, Y''_j]$, where $F_j := \bigcup_{i < j} (R'_i \cup R''_i)$.

    It follows from \ref{(A2)} and the fact that 
    \[V(H_{X,j}) = |X| = |Y| = V(H_{Y,j})\]
    that $H_{X,j}$ and $H_{Y,j}$ have the same average degree, so we have
    \begin{align*}
        \Delta(H_{X,j} \cup H_{Y,j}) - \delta(H_{X,j} \cup H_{Y,j}) &\leq \left(\Delta(H_{X,j}) - \delta(H_{X,j})\right) + \left(\Delta(H_{Y,j}) - \delta(H_{Y,j})\right) \\
        &\leq \frac{\gamma\rho d}{2K} \\
        &\leq \frac{\rho d}{K} - C
    \end{align*}
    by \ref{(A1)} and the assumption that $K \leq \frac{\gamma \rho d}{2}$. This means that $f_j'$ and $f''_j$ both take values in~$\left[(1-\gamma)\frac{\rho d}{K}, \frac{\rho d}{K}\right]$. Furthermore, we have by \ref{(A2)} that
    \begin{align*}
        f'_j(X'_j) &= \left(C + \Delta(H_{X,j} \cup H_{Y,j})\right)|X'_j| - e(H_{X,j}) \\
        &= \left(C + \Delta(H_{X,j} \cup H_{Y,j})\right)|Y'_j| - e(H_{Y,j}) \\
        &= f'_j(Y'_j),
    \end{align*}
    and a similar calculation shows $f''_j(X''_j) = f''_j(Y''_j)$. Also recall that $\frac{\rho}{K} \leq \alpha$ by assumption, and that $\Delta\left(F_j\right) \leq \rho d$ by the inductive hypothesis. Therefore, by \ref{(A3)} and the fact that $G[X,Y]$ is~$(d,\rho,\alpha,\gamma)$-good, we see that $\left(G\setminus F_j\right)[X'_j, Y'_j]$ has an~$f'_j$-factor $R'_j$, and $\left(G\setminus F_j\right)[X''_j, Y''_j]$ has an~$f''_j$-factor $R''_j$. Now by design, we have that~$R'_j \cup R''_j$ is edge-disjoint from ${F_j = \bigcup_{i < j} (R'_i \cup R''_i)}$, and that~$\Delta(R'_j \cup R''_j) \leq \frac{\rho d}{K}$. We also have that $H_{X,j} \cup H_{Y,j} \cup R'_j \cup R''_j$ is a bipartite $\left(C + \Delta(H_{X,j} \cup H_{Y,j})\right)$-regular spanning subgraph of $G$, with bipartition $\{X'_j \cup Y''_j, X''_j \cup Y'_j\}$. This completes the inductive step.

    At last, we define $H := G \setminus \bigcup_{j=1}^K (H_{X,j} \cup H_{Y,j} \cup R'_j \cup R''_j)$. Since $G$ is regular, and the edge-disjoint regular graphs $H_{X,j} \cup H_{Y,j} \cup R'_j \cup R''_j$ cover all of the edges in $G[X] \cup G[Y]$, we see that $H$ is a regular spanning subgraph of $G[X,Y]$. Thus $\{H_{X,j} \cup H_{Y,j} \cup R'_j \cup R''_j\}_{j=1}^K \cup \{H\}$ is a decomposition of $G$ into $K+1$ regular bipartite spanning subgraphs, as desired.
\end{proof}

\section{Proof of our main result}\label{section:main proof}
\indent

We are now ready to prove our main theorem.
\begin{proof}[Proof of \Cref{main-thm}]
Let $d$ be a sufficiently large positive integer, and let $n \geq d+1$ be an even integer. Let $G$ be an $(n,d,\lambda)$-graph on $[n]$ with $\lambda \leq \frac{d}{12}$. First, we take a balanced bipartition $\{X,Y\}$ of $[n]$ as in \Cref{lem:almost regular bipartition}. Note that $e(G[X]) = e(G[Y])$. Indeed, since $G$ is regular and $|X| = |Y| = \frac{n}{2}$, we have
\begin{align*}
    2e(G[X]) + e(G[X,Y]) &= d|X| = d|Y| = 2e(G[Y]) + e(G[X,Y]).
\end{align*}

Let $M := \lfloor \log_2 d \rfloor - 18$. For $1 \leq j \leq M$, we iteratively obtain a sequence of pairs $(H_{X,j}, H_{Y,j})$ of edge-disjoint bipartite spanning subgraphs $H_{X,j} \subseteq G[X]$ and $H_{Y,j} \subseteq G[Y]$. For each $U \in \{X,Y\}$, we set $G_{U,0} := G[U]$, and for $1 \leq j \leq M$, we set $G_{U,j} := G_{U,j-1} \setminus H_{U,j}$. For $0 \leq j \leq M$, we also define $d_j := \frac{d}{2^{j+1}}$ and $\epsilon_j := 40d_j^{-1/3}$.
\begin{claim}\label{claim:iteration works}
    We can ensure the following properties for $0 \leq j \leq M$:
    \begin{enumerate}[label=(I\arabic*)]
    \item\label{(I1)} if $j \geq 1$, then for each $U \in \{X,Y\}$, we have $\Delta(H_{U,j}) - \delta(H_{U,j}) \leq 70d_j^{2/3}$;
    \item\label{(I2)} if $j \geq 1$, then $e(H_{X,j}) = e(H_{Y,j})$;
    \item\label{(I3)} if $j \geq 1$, then for each $U \in \{X,Y\}$, the bipartite graph $H_{U,j}$ has a bipartition which is $\frac{d}{5}$-good with respect to $G[X,Y]$;
	\item\label{(I4)} for each $U \in \{X,Y\}$, we have $\deg_{G_{U,j}}(v) = (1 \pm \epsilon_{j})d_j$ for every $v \in U$.
    \end{enumerate}
\end{claim}

\begin{proof}[Proof of \Cref{claim:iteration works}]
    We observe that \ref{(I1)}--\ref{(I4)} are satisfied for $j=0$ already. Indeed, \ref{(I1)}--\ref{(I3)} hold trivially, and by \Cref{lem:almost regular bipartition}, for each $U \in \{X,Y\}$, we have
    \[\deg_{G_{U,0}}(v) = \frac{d}{2} \pm d^{2/3} = (1 \pm \epsilon_0)d_0\]
    for every $v \in U$. 
    Now suppose that, for some $j \geq 0$, we have the graphs~$H_{X,1}, \ldots, H_{X,j}$ and~$H_{Y,1}, \ldots, H_{Y,j}$ as claimed. If $j < M$, then we now wish to obtain $H_{X,j+1}$ and $H_{Y,j+1}$.

    We first obtain, for each $U \in \{X,Y\}$, an induced bipartite spanning subgraph $H_{U,j+1}'$ of $G_{U,j}$ as in \Cref{lem:good bisection of one half}, with $\epsilon_j$, $d_j$, $G[X,Y]$, $U$, $G_{U,j}$, $H'_{U,j+1}$ in place of $\epsilon$, $d'$, $H$, $X$, $G_X$, $H_X$, respectively.
    We note that $d_j \geq \frac{d}{2^M} \geq 2^{18}$, so then $\epsilon_j \leq 40 \cdot 2^{-18/3} = \frac{5}{8}$, and our application of \Cref{lem:good bisection of one half} is justified. Specifically, by \ref{(L1)} for $H'_{U,j+1}$ and \ref{(I4)} for $G_{U,j}$, we have that 
    \begin{equation}\label{eq:L1 explicit}
        \deg_{H_{U,j+1}'}(v) = \left(1 \pm 2d_j^{-1/3}\right)\frac{\deg_{G_{U,j}}(v)}{2} = \left(1 \pm 2d_j^{-{1/3}}\right)\frac{(1 \pm \epsilon_j)d_j}{2}
    \end{equation}
    for every $v \in U$.

    Next, let $m := \min\{e(H'_{X,j+1}), e(H'_{Y,j+1})\}$. For each $U \in \{X,Y\}$, let $H_{U,j+1}$ be an $m$-edge spanning subgraph of $H'_{U,j+1}$ as in \Cref{lem:subgraph with m edges}, with $H'_{U,j+1}$, $H_{U,j+1}$ in place of $H$, $H'$, respectively. By \ref{(S1)} for $H_{U,j+1}$ and \eqref{eq:L1 explicit}, we have
    \begin{align*}
        \Delta(H_{U,j+1}) -\delta(H_{U,j+1}) &\leq \left(1 + 2d_j^{-1/3}\right)\frac{(1 + \epsilon_j)d_j}{2} - \left(1 - 2d_j^{-1/3}\right)\frac{(1 - \epsilon_j)d_j}{2} + 2 \\
        &= 42d_j^{2/3} + 2 \\
        &< 70d_{j+1}^{2/3}.
    \end{align*}
    Thus each $H_{U,j+1}$ satisfies \ref{(I1)}. Now $e(H_{X,j+1}) = m= e(H_{Y,j+1})$, so \ref{(I2)} is satisfied as well. Also, since each $H_{U,j+1}$ is a spanning subgraph
    of $H'_{U,j+1}$, we see that \ref{(I3)} follows from \ref{(L2)} for $H'_{U,j+1}$.
    
    It remains to verify that each $G_{U,j+1}= G_{U,j} \setminus H_{U,j+1}$ satisfies \ref{(I4)}. We first observe that
    \begin{equation}\label{eq:I4 verification}
        \deg_{G_{U,j} \setminus H'_{U,j+1}}(v) \leq \deg_{G_{U,j+1}}(v) \leq \deg_{G_{U,j} \setminus H'_{U,j+1}}(v) + \Delta(H'_{U,j+1} \setminus H_{U,j+1})
    \end{equation}
    for each $U \in \{X,Y\}$ and every $v \in U$. Next, note that that
    \[e(H'_{U,j+1}) = \frac{1}{2}\sum_{v \in U} \deg_{H'_{U,j+1}}(v) = \frac{1}{2}\sum_{v \in U}\left(1 \pm 2d_j^{-1/3}\right)\frac{\deg_{G_{U,j}}(v)}{2} = \left(1 \pm 2d_j^{-1/3}\right)\frac{e(G_{U,j})}{2}\]
    for each $U \in \{X,Y\}$ by \eqref{eq:L1 explicit}. Also, recall that $e(G[X]) = e(G[Y])$, so by~\ref{(I2)} and induction, we have that $e(G_{X,j}) = e(G_{Y,j})$. It follows that for each $U \in \{X,Y\}$, we have by our choice of $m$ that
    \[\left(1 - 2d_j^{-1/3}\right)\frac{e(G_{U,j})}{2} \leq m \leq e(H'_{U,j+1}) \leq \left(1 + 2d_j^{-1/3}\right)\frac{e(G_{U,j})}{2}.\]
    Now by~\ref{(S2)} and \eqref{eq:L1 explicit}, we have
    \begin{align*}
        \Delta(H'_{U,j+1} \setminus H_{U,j+1}) &\leq \left(\Delta(H'_{U,j+1})+1\right)\frac{e(H'_{U,j+1})-m}{e(H'_{U,j+1})} + 1 \\
        &\leq \left(\Delta(H'_{U,j+1})+1\right)\frac{4d_j^{-1/3}}{1+2d_j^{-1/3}} + 1 \\
        &\leq 2(1 + \epsilon_j)d_j^{2/3} + 3.
    \end{align*}
    Now using the upper and lower bounds \eqref{eq:I4 verification} on $\deg_{G_{U,j+1}}(v)$ observed above, we have by \eqref{eq:L1 explicit} and \ref{(I4)} for $G_{U,j}$ that
        \begin{align*}
            \deg_{G_{U,j+1}}(v) &\leq (1 + 2d_j^{-1/3})\frac{(1+\epsilon_j)d_j}{2} + 2(1+\epsilon_j)d_j^{2/3} + 3 \\
            &= \frac{d_j}{2} + 23d_j^{2/3} + 120d_j^{1/3} + 3 \\
            &< \frac{d_j}{2} + 25d_j^{2/3} \\
            &< d_{j+1} + 40d_{j+1}^{2/3} \\
            &= (1 + \epsilon_{j+1})d_{j+1}
        \end{align*}
        since $d_j \geq 2^{18}$, and also
        \[\deg_{G_{U,j+1}}(v) \geq \left(1-2d_j^{-1/3}\right)\frac{(1-\epsilon_j)d_j}{2} \geq \frac{d_j}{2} - 21d_j^{2/3} \geq (1-\epsilon_{j+1})d_{j+1}.\]
        Thus $G_{U,j+1}$ satisfies \ref{(I4)}, and the proof of \Cref{claim:iteration works} is complete.
    \end{proof}

We are now left with graphs $G_{X,M}$ and $G_{Y,M}$ of maximum degree at most $(1+\epsilon_M)\frac{d}{2^{M+1}} \leq 2^{19}$. By \Cref{lem:cleanup}, there exist spanning bipartite decompositions $\{H'_{X,M+j}\}_{j=1}^{27}$ and $\{H'_{Y,M+j}\}_{j=1}^{27}$ of~$G_{X,M}$ and~$G_{Y,M}$, respectively, such that each $H'_{X,j}$ and each $H'_{Y,j}$ has a bipartition which is $\frac{d}{5}$-good with respect to $G[X,Y]$. By \Cref{lem:same size refinements} applied to $\{E(H_{X,M+j})\}_{j=1}^{27}$ and $\{E(H_{Y,M+j})\}_{j=1}^{27}$, there exist $t \leq 53$ and respective spanning bipartite decompositions $\{H_{X,M+j}\}_{j=1}^t$ and $\{H_{X,M+j}\}_{j=1}^t$ of $G_{X,M}$ and $G_{Y,M}$ such that for each $1 \leq j \leq t$, $H_{U,M+j}$ is a spanning subgraph of $H'_{U,j'}$ for some $j'$, and $e(H_{X,M+j}) = e(H_{Y,M+j})$.

Finally, we claim that $\{H_{X,j}\}_{j=1}^{M+t}$ and $\{H_{Y,j}\}_{j=1}^{M+t}$ satisfy the hypotheses of \Cref{lem:regularization} with
\[{\rho := \frac{1}{200}}, \quad \alpha := \frac{1}{200}, \quad \gamma := \frac{1}{30}, \quad \text{and} \quad K := M + t\]
if $d$ is sufficiently large, so $G$ admits a decomposition into at most $M+t+1 \leq \log_2 d + 36$ regular bipartite spanning subgraphs. Indeed, $G$ is~$(d,\rho, \alpha,\gamma)$-good by \Cref{lem:good induced bipartite subgraphs of expanders}. We also have~$\frac{\rho}{\alpha} \leq K \leq \frac{\gamma \rho d}{2}$. For~$1 \leq j \leq M$ and~$U \in \{X,Y\}$, we have~$\Delta(H_{U,j}) - \delta(H_{U,j}) \leq 70d^{2/3} \leq \frac{\gamma \rho d}{2K}$ by~\ref{(I1)}. This, together with~\ref{(I2)}--\ref{(I3)}, verifies~\ref{(A1)}--\ref{(A3)} for~$1 \leq j \leq M$. For~$M + 1 \leq j \leq M+t$ and~$U \in \{X,Y\}$, we have~$\Delta(H_{U,j}) - \delta(H_{U,j}) \leq \Delta(G_{U,M}) \leq 2^{19} \leq \frac{\gamma \rho d}{4K}$, so \ref{(A1)} is satisfied. We have already observed that \ref{(A2)} is satisfied, and \ref{(A3)} follows from the fact that each $H_{U,j}$ with $M+1 \leq j \leq M+t$ is a spanning subgraph of some $H'_{U,j'}$. This concludes the proof of \Cref{main-thm}.
\end{proof}

\section*{Acknowledgements}
We would like to thank the organizers and the participants of the Desert Discrete Math Workshop 2024, where this research project was initiated and a part of it was completed.

\bibliographystyle{abbrv}
	\bibliography{literature}
\end{document}